\documentclass[12pt]{amsart}
\usepackage[nobysame,abbrev,alphabetic]{amsrefs}
\usepackage{amssymb}
\usepackage[all]{xy}

\DeclareMathOperator{\cd}{cd}

\DeclareMathOperator{\Coker}{Coker}

\DeclareMathOperator{\Hom}{Hom}
\DeclareMathOperator{\id}{id}
\DeclareMathOperator{\Img}{Im}

\DeclareMathOperator{\invlim}{\varprojlim}
\DeclareMathOperator{\Ker}{Ker}

\DeclareMathOperator{\res}{res}

\DeclareMathOperator{\GL}{GL}
\DeclareMathOperator{\trg}{trg}

\begin{document}

\newtheorem{thm}{Theorem}[section]
\newtheorem{cor}[thm]{Corollary}
\newtheorem{lem}[thm]{Lemma}
\newtheorem{prop}[thm]{Proposition}
\newtheorem{defin}[thm]{Definition}
\newtheorem{exam}[thm]{Example}
\newtheorem{examples}[thm]{Examples}
\newtheorem{rem}[thm]{Remark}
\newtheorem*{thmA}{Theorem A}
\newtheorem*{thmA'}{Theorem A'}
\newtheorem*{thmB}{Theorem B}
\swapnumbers
\newtheorem{rems}[thm]{Remarks}
\newtheorem*{acknowledgment}{Acknowledgment}
\numberwithin{equation}{section}

\newcommand{\dec}{\mathrm{dec}}
\newcommand{\dirlim}{\varinjlim}
\newcommand{\nek}{,\ldots,}
\newcommand{\inv}{^{-1}}
\newcommand{\isom}{\cong}
\newcommand{\Massey}{\mathrm{Massey}}
\newcommand{\ndiv}{\hbox{$\,\not|\,$}}
\newcommand{\pr}{\mathrm{pr}}
\newcommand{\tensor}{\otimes}
\newcommand{\U}{\mathrm{U}}
\newcommand{\alp}{\alpha}
\newcommand{\gam}{\gamma}
\newcommand{\del}{\delta}
\newcommand{\eps}{\epsilon}
\newcommand{\lam}{\lambda}
\newcommand{\Lam}{\Lambda}
\newcommand{\sig}{\sigma}
\newcommand{\dbF}{\mathbb{F}}
\newcommand{\dbZ}{\mathbb{Z}}

\title[The Zassenhaus Filtration]
{The Zassenhaus filtration,  Massey Products, and Representations of Profinite Groups}

\author{Ido Efrat}
\address{Mathematics Department\\
Ben-Gurion University of the Negev\\
P.O.\ Box 653, Be'er-Sheva 84105\\
Israel} \email{efrat@math.bgu.ac.il}

\thanks{This work was supported by the Israel Science Foundation (grants No.\ 23/09 and No.\ 152/13).}


\maketitle
\begin{abstract}
We consider the $p$-Zassenhaus filtration $(G_n)$ of a profinite group $G$.
Suppose that $G=S/N$ for a free profinite group $S$ and a normal subgroup $N$ of $S$ contained in $S_n$.
Under a cohomological assumption on the $n$-fold Massey products (which holds, e.g., if $G$ has $p$-cohomological dimension $\leq1$),
we prove that $G_{n+1}$ is the intersection of all
kernels of upper-triangular unipotent $(n+1)$-dimensional representations of $G$ over $\dbF_p$.
This extends earlier results by Min\'a\v c, Spira, and the author on the structure of absolute Galois groups of fields.
\end{abstract}

\section{Introduction}
Let $p$ be a fixed prime number.
The \textbf{$p$-Zassenhaus filtration} of a profinite group $G$ is the fastest  descending sequence of closed subgroups
$G_n$, $n=1,2\nek$ of $G$ such that $G_1=G$, $G_i^p\leq G_{ip}$, and $[G_i,G_j]\leq G_{i+j}$ for $i,j\geq1$.
Thus $G_n=G_{\lceil n/p\rceil}^p\prod_{i+j=n}[G_i,G_j]$ for $n\geq2$.
Here given closed subgroups $H,K$ of $G$, we write $[H,K]$ (resp., $H^p$) for the closed subgroup
generated by all commutators  $[\sig,\tau]$ (resp., $p$-th powers $\sig^p$), with $\sig\in H$, $\tau\in K$.
This filtration has been studied since the middle of the 20th century both from the group-algebra and Lie algebra
viewpoints (\cite{Jennings41}, \cite{Lazard54}, \cite{Lazard65}, \cite{Quillen68}, \cite{Zassenhaus39},  \cite{DixonDuSautoyMannSegal99}*{Ch.\ 11--12}),
and had important Galois-theoretic applications, e.g., to the Golod--\v Safarevi\v c problem \cite{Koch02}*{\S7.7},
mild groups, and Galois groups of restricted ramification (\cite{Labute05}, \cite{LabuteMinac11},
\cite{Morishita04}, \cite{Vogel05}).

In this paper we interpret the $p$-Zassenhaus filtration from the viewpoint of linear representations over $\dbF_p$,
and use this to explain and generalize several known facts on the structure of absolute Galois groups of fields.
Our first main result characterizes the filtration for free pro-$p$ groups:

\begin{thmA}
Let $\bar S$ be a free pro-$p$ group and $n\geq1$.
Then $\bar S_n$ is the intersection of all kernels of linear representations
$\rho\colon \bar S\to \GL_n(\dbF_p)$.
\end{thmA}

We may clearly replace here $\GL_n(\dbF_p)$ by its $p$-Sylow subgroup $\U_n(\dbF_p)$, consisting of
all upper-triangular $n\times n$ unipotent matrices over $\dbF_p$.
In this reformulation, Theorem A extends to general profinite groups $G$, whose defining relations lie sufficiently
high in the $p$-Zassenhaus filtration, and which satisfy a certain cohomological condition, to be explained below.

Specifically, let $n\geq2$, and suppose that $G$ can be presented as $S/N$ for a free \textsl{profinite} group $S$ and
a closed normal subgroup $N$ of $S$ contained in $S_n$ (this can be slightly relaxed - see Remark \ref{relaxed assumption}).
The cohomological condition we assume is on the \textsl{$n$-fold Massey products} in the
mod-$p$ profinite cohomology ring $H^\bullet(G/G_n)=H^\bullet(G/G_n,\dbZ/p)$.
These are maps $H^1(G/G_n)^n\to H^2(G/G_n)$,
that generalize the usual cup product (which is essentially the case $n=2$).
According to their general construction (recalled in \S\ref{section on Massey products}),
the Massey products are multi-valued maps.
Yet, it was shown by Vogel \cite{Vogel05}  that in our situation they are well-defined maps
(see  \S\ref{section on Massey products for inhomogenous cochains}).
Moreover, they are multi-linear.
Let $H^2(G/G_n)_{n-\Massey}$ be the subgroup of $H^2(G/G_n)$ generated by the image of this map.
We say that $G$ satisfies the \textbf{$n$-th Massey kernel condition} if the
kernel of the inflation map $\inf\colon H^2(G/G_n)_{n-\Massey}\to H^2(G)$
is generated by $n$-fold Massey products from $H^1(G/G_n)^n$.
When $n=1$ we declare the condition to be true by definition.
We prove:

\begin{thmA'}
If $G$ satisfies the $n$-th Massey kernel condition, then $G_{n+1}$ is the intersection of all
kernels of representations $G\to\U_{n+1}(\dbF_p)$.
\end{thmA'}

The assumptions of Theorem A' are satisfied when $G$ has $p$-cohomological dimension $\leq1$
(Corollary \ref{cd leq 1}).
Hence Theorem A is a special case of Theorem A'.

Theorem A' is  elementary for $n=1$.
For $n=2$, $p=2$ it was proved in \cite{EfratMinac11a}*{Cor.\ 11.3}, generalizing a result of Min\'a\v c and Spira
\cite{MinacSpira96}*{Cor.\ 2.18}  (see also \cite{Thong12}).
For $n=2$, $p>2$ it was proved in \cite{EfratMinac11b}*{Example 9.5(1)}.
Moreover, it has the following remarkable Galois-theoretic application (which was obtained in \cite{MinacSpira96}*{Cor.\ 2.18} and  \cite{EfratMinac11b}*{Example 9.5(1)}; see \S\ref{section on examples}):
Assume that  $G=G_K$ is the absolute Galois group of a field $K$ containing a root of unity of order $p$.
Then $G_3$ is the intersection of all normal open subgroups $M$ of $G$ such that $G/M$ embeds
in $\U_3(\dbF_p)$.
The proof of this fact uses the Merkurjev--Suslin theorem \cite{MerkurjevSuslin82}.
Note that $\U_3(\dbF_p)$ is either the dihedral group $D_4$ (when $p=2$)
or the Heisenberg group $H_{p^3}$ (when $p>2$).

The proof of Theorem A' is based on an alternative description of the $p$-Zassenhaus filtration
in terms of Magnus algebra $\dbF_p\langle\langle X_A\rangle\rangle$ of formal power series over $\dbF_p$
in non-commuting variables $X_a$, where $a$ ranges over a basis $A$ of $S$.
The \textsl{Magnus homomorphism} $\Lam\colon S\to \dbF_p\langle\langle X_A\rangle\rangle^\times$
is defined by $a\mapsto 1+X_a$  (see \S\ref{section on free profinite groups}).
Then $S_n$ is the preimage under $\Lam$
of the multiplicative group of all formal power series $1+\sum_{|w|\geq n}c_wX_w$ (Proposition \ref{Jennings-Brauer}).

Behind the proof of Theorem A' is also a key observation from \cite{EfratMinac11b}, relating
such intersection results with duality principles between profinite groups and subgroups of second
cohomology groups (see Proposition \ref{EM2 Prop 32} for the precise statement).
In our case we consider the subgroup  $H^2(S/S_n)_{n-\Massey}$ of $H^2(S/S_n)$ and prove:

\begin{thmB}
\begin{enumerate}
\item[(a)]
There is a natural perfect pairing
\[
S_n/S_{n+1}\times H^2(S/S_n)_{n-\Massey}\to \dbZ/p.
\]
\item[(b)]
There is a natural exact sequence
\[
0\to H^2(S/S_n)_{n-\Massey}\hookrightarrow H^2(S/S_n)\to H^2(S/S_{n+1}).
\]
\end{enumerate}
\end{thmB}
See Corollary \ref{thm B for G} for a generalization of these facts to profinite groups $G$ as above.

Finally, a method of Dwyer  \cite{Dwyer75} expresses the central extensions corresponding to elements
of $H^2(S/S_n)_{n-\Massey}$ by means of linear representations into $\U_n(\dbZ/p)$ (see \S\ref{section on
Massey products for inhomogenous cochains}), leading to Theorem A'.

In the situation of Theorem A, where $G$ is a free pro-$p$ group, the cohomological picture is of course considerably simpler.
Indeed, the subsequent papers by Min\'a\v c and T\^an \cite{MinacTan13} and the author \cite{Efrat13} (independently) give a more direct proof of Theorem A based on the Magnus theory.
Further, Theorem A is extended there to various related filtrations of free pro-$p$ groups.

Fundamental connections between the $p$-Zassenhaus filtration and Massey products
were earlier observed and studied in the works of Dwyer \cite{Dwyer75},  Morishita \cite{Morishita04}, \cite{Morishita12} and Vogel  \cite{Vogel05};  see also \cite{Stallings65}*{\S5}.
In fact, in the case $p=2$, $n=3$, related connections were earlier studied in \cite{GaoLeepMinacSmith03}
under a different terminology (this was pointed out to me by J\'an Min\'a\v c).
Among the other recent important works on Massey products in Galois theory and arithmetic geometry are
those by Sharifi \cite{Sharifi07}, Wickelgren, Hopkins (\cite{Wickelgren09}, \cite{Wickelgren12}, \cite{HopkinsWickelgren12}), and G\"artner \cite{Gaertner12}.

I warmly thank J\'an Min\'a\v c for many very inspiring discussions which motivated the present work
and for his comments.
I also thank the referee for his/her insightful remarks and suggestions, as well as to A.\ Lubotzky, M.\ Schacher, N.D.\ T\^an, and K.\ Wickelgren for their comments.

\section{Preliminaries on bilinear pairings}
First we recall some terminology and facts on bilinear maps.
We fix a commutative ring $R$.
A bilinear map $(\cdot,\cdot)\colon A\times B\to R$ of $R$-modules $A$ and $B$ is
\textbf{non-degenerate} (resp., a \textbf{perfect pairing}) if the induced $R$-module
homomorphisms $A\to\Hom_R(B,R)$ and $B\to\Hom_R(A,R)$ are injective (resp., bijective).
We say that a diagram of bilinear pairings and $R$-homomorphisms
\begin{equation}
\label{cd of pairings}
\xymatrix{
A & *-<3pc>{\times} & B\ar[d]^{\beta}\ar[r]^{(\cdot,\cdot)} & R\ar@{=}[d] \\
A'\ar[u]^{\alp} & *-<3pc>{\times} & B'\ar[r]^{(\cdot,\cdot)'} & R
}
\end{equation}
\textbf{commutes}  if $(\alp(a'),b)=(a',\beta(b))'$ for every $a'\in A'$ and $b\in B$.

The proofs of the next two lemmas are straightforward:

\begin{lem}
\label{limits}
Let $(I,\leq)$ be a filtered set and for every $i\in I$ let $A_i\times B_i\to R$ be a perfect pairing.
Let $\alp_{ij}\colon A_j\to A_i$, $\beta_{ij}\colon B_i\to B_j$, for $i,j\in I$ with $i\leq j$, be homomorphisms
which commute with the pairings, and are compatible in the natural sense.
Then there is an induced perfect pairing  $\invlim A_i\times\dirlim B_i\to R$.
\end{lem}

\begin{lem}
\label{pairing of images}
Suppose that (\ref{cd of pairings}) commutes and the induced maps $A\to\Hom_R(B,R)$, $B'\to\Hom_R(A',R)$
are injective.
Then (\ref{cd of pairings}) induces a non-degenerate bilinear map $\Img(\alp)\times\Img(\beta)\to R$.
\end{lem}

\begin{lem}
\label{pairing of Coker and Ker}
Suppose that (\ref{cd of pairings}) commutes, $(\cdot,\cdot)$
is non-degenerate, the induced map $A'\mapsto\Hom_R(B',R)$ is surjective, and the $R$-module $B'$ is semi-simple.
Then (\ref{cd of pairings}) induces  a non-degenerate
bilinear map $\Coker(\alp)\times\Ker(\beta)\to R$.
\end{lem}
\begin{proof}
Set  $A''=\Coker(\alp)$ and $B''=\Ker(\beta)$.
The existence of an induced well-defined bilinear map $A''\times B''\to R$ is straightforward from the commutativity of (\ref{cd of pairings}).
Since $(\cdot,\cdot)$ is non-degenerate, the induced map $B''\to\Hom_R(A'',R)$ is injective.

Finally, there is a commutative diagram
\begin{equation}
\label{cd for snake}
\xymatrix{
 A'\ar[r]^{\alp}\ar@{->>}[d] &  A\ar[r]\ar@{>->}[d] & A''\ar[r]\ar[d] & 0 \\
 \Hom_R(B',R)\ar[r] &  \Hom_R(B,R)\ar[r] & \Hom_R(B'',R) & \quad,
 }
\end{equation}
where the upper row is clearly exact, and the lower row is a complex.
Moreover, the exact sequence
\[
0\to B''\to B\to\Img(\beta)\to 0
\]
gives an exact sequence
\begin{equation}
\label{efg}
0\to\Hom_R(\Img(\beta),R)\to\Hom_R(B,R)\to\Hom_R(B'',R).
\end{equation}
By the semi-simplicity assumption, $\Img(\beta)$ is a direct summand of $B'$.
Hence $\Hom_R(\Img(\beta),R)$ is a direct summand of $\Hom_R(B',R)$, so by (\ref{efg}), the lower row in (\ref{cd for snake}) is in fact exact.
Now the snake lemma shows that the right vertical map in (\ref{cd for snake}) is injective.
\end{proof}

\section{Massey products}
\label{section on Massey products}
Recall that a \textbf{differential $\dbZ$-graded algebra}  (DGA) over a ring $R$ is a graded $R$-algebra
$C^\bullet=\bigoplus_{r\in\dbZ}C^r$ equipped with homomorphisms
$\partial^r\colon C^r\to C^{r+1}$ such that $(C^\bullet,\bigoplus_{r\in\dbZ}\partial^r)$ is a complex satisfying the \textsl{Leibnitz rule}
$\partial^{r+s}(ab)=\partial^r(a)b+(-1)^ra\partial^s(b)$ for  $a\in C^r$, $b\in C^s$.
Let $H^i=\Ker(\partial^i)/\Img(\partial^{i-1})$ be the $i$-th cohomology group of $C^\bullet$.
For an $i$-cocycle $c$ let $[c]$ be its cohomology class in $H^i$.
Thus $[c]=[c']$ will mean that $c,c'$ are homogenous cocycles of equal degree and are cohomologous.

We fix  an integer $n\geq 2$.
We consider systems $c_{ij}\in C^1$, where $1\leq i\leq j\leq n$ and $(i,j)\neq(1,n)$.
For any $i,j$ satisfying $1\leq i\leq j\leq n$ (including $(i,j)=(1,n)$) we may define
\[
\widetilde c_{ij}=-\sum_{r=i}^{j-1}c_{ir}c_{r+1,j}\in C^2.
\]
One says that $(c_{ij})$ is a  \textbf{defining system of size $n$} in $C^\bullet$ if
$\widetilde c_{ij}=\partial c_{ij}$ for every $1\leq i\leq j\leq n$ with $(i,j)\neq(1,n)$.
We also say that the defining system $(c_{ij})$ is \textbf{on $c_{11}\nek c_{nn}$}.
Note that then $c_{ii}$ is a $1$-cocycle, $i=1,2\nek n$.
Further, $\widetilde c_{1n}$ is a $2$-cocycle (\cite{Kraines66}*{p.\ 432}, \cite{Fenn83}*{p.\ 233}, \cite{EfratMassey}).

\begin{lem}
\label{dltzf}
Suppose that for every list of cocycles $c_1\nek c_n\in C^1$ there is a defining system of size $n$ on $c_1\nek c_n$  in $C^\bullet$.
Let $1\leq k<n$.
Then for every list of cocycles $c_1\nek c_k\in C^1$ there is a defining system of size $k$ on $c_1\nek c_k$ such that in addition $[\widetilde c_{1k}]=0$.
\end{lem}
\begin{proof}
Choose arbitrary cocycles $c_{k+1}\nek c_n\in C^1$ and a defining system $(c_{ij})$ on $c_1\nek c_k,c_{k+1}\nek c_n$.
Then $(c_{ij})$, $1\leq i\leq j\leq k$, $(i,j)\neq(1,k)$, is a defining system on $c_1\nek c_k$.
Moreover, $\widetilde c_{1k}=\partial c_{1k}$, so $[\widetilde c_{1k}]=0$.
\end{proof}

The main objective of this section is Proposition \ref{uniqueness} below.
Its proof will be based on the following fact from \cite{Fenn83}*{Lemma 6.2.7},  which is a variant of \cite{Kraines66}*{Lemma 20} (with different sign conventions).
We note that while the latter results are stated in a topological setting, their proofs are at the level of general DGAs.
A self-contained exposition of these results,  as well as of the multi-linearity of the Massey product (see below), is given in \cite{EfratMassey}.

\begin{prop}[Fenn, Kraines]
\label{Kraines}
Assume that for every $1\leq k<n$ and every defining system $(d_{ij})$ of size $k$ in $C^\bullet$ one has $[\widetilde d_{1k}]=0$.
Let $(c_{ij}),(c'_{ij})$ be defining systems of size $n$ in $C^\bullet$ with $[c_{ii}]=[c'_{ii}]$, $i=1,2\nek n$.
Then $[\widetilde c_{1n}]=[\widetilde c'_{1n}]$.
\end{prop}

\begin{prop}
\label{uniqueness}
Suppose that for every list of cocycles $c_1\nek c_n\in C^1$ there is a defining system on $c_1\nek c_n$ in $C^\bullet$.
Then for any two defining systems $(c_{ij}),(c'_{ij})$ of size $n$ with $[c_{ii}]=[c'_{ii}]$, $i=1,2\nek n$, one has $[\widetilde c_{1n}]=[\widetilde c'_{1n}]$.
\end{prop}
\begin{proof}
We argue by induction on $n$.
For $n=2$ this is immediate from the definition of $\widetilde c_{12},\widetilde c'_{12}$.

For arbitrary $n$, the assumption holds also for smaller values of $n$, by Lemma \ref{dltzf}.
We obtain the last sentence of the proposition for smaller values of $n$.
We need to verify the assumption of Proposition \ref{Kraines}.
So let $1\leq k<n$, and consider a defining system $(d_{ij})$ of size $k$ in $C^\bullet$.
Lemma \ref{dltzf} again yields a defining system $(d'_{ij})$ of size $k$ on $d_{11}\nek d_{kk}$ such that in addition $[\widetilde d'_{1k}]=0$.
By the inductive hypothesis,  $[\widetilde d_{1k}]=[\widetilde d'_{1k}]=0$, as required.
\end{proof}

Under the assumptions of Proposition \ref{uniqueness}, define the \textbf{Massey product}
\begin{equation}
\langle\cdot\nek\cdot\rangle\colon H^1\times\cdots\times H^1\to H^2,
\end{equation}
as follows:
given cocycles  $c_1\nek c_n\in C^1$ we take a defining system
$(c_{ij})$ in $C^\bullet$ with $c_{11}=c_1\nek c_{nn}=c_n$.
As remarked above, $\widetilde c_{1n}$ is a $2$-cocycle.
We set $\langle[c_1]\nek[c_n]\rangle=[\widetilde c_{1n}]$.
By Proposition \ref{uniqueness}, this map is well-defined.
Moreover, it is multi-linear (\cite{Fenn83}*{Lemma 6.2.4}, \cite{EfratMassey}).

\section{Formal power series}
\label{section on formal power series}
Let $(A,\leq)$ be a set, considered as an alphabet, and $A^*$ be the set of all finite words on $A$.
We write $\emptyset$ for the empty word, and $|w|$ for the length of the word $w$.

We fix a commutative unital ring $R$ and noncommuting variables $X_a$, $a\in A$.
The collection $X_A$ of all formal expressions  $X_w=X_{a_1}\cdots X_{a_n}$, with  $w=(a_1\nek a_n)\in A^*$,
forms a monoid under concatenation.
Let $R\langle\langle X_A\rangle\rangle$ be the ring of all formal power series $\sum_{w\in A^*}c_wX_w$,
with $c_w\in R$.
We denote by $R\langle\langle X_A\rangle\rangle^\times$  its multiplicative group.

For a positive integer $n$ let $V_{n,R}$ be the subset of $R\langle\langle X_A\rangle\rangle^\times$
consisting of all power series $\sum_wc_wX_w$ such that $c_\emptyset=1$ and $c_w=0$ for every $w$ with $1\leq |w|<n$.
Note that any $1+\alp\in V_{n,R}$ has an inverse $\sum_{k=0}^\infty(-1)^k\alp^k$ in $V_{n,R}$,
so $V_{n,R}$ is a subgroup of $R\langle\langle X_A\rangle\rangle^\times$.
We write $V_{n,R}^m$ for the subgroup of $V_{n,R}$ generated by all $m$-th powers, and $[V_{n,R},V_{k,R}]$ for the subgroup
of $V_{1,R}$ generated by all commutators $[\gam,\del]=\gam\inv\del\inv\gam\del$,
with $\gam\in V_{n,R}$, $\del\in V_{k,R}$.

\begin{lem}
\label{filtration of power series}
\begin{enumerate}
\item[(a)]
$[V_{n,R},V_{k,R}]\leq V_{n+k,R}$;
\item[(b)]
If $mR=0$, then $V_{n,R}^m\leq V_{n+1,R}$.
\end{enumerate}
\end{lem}
\begin{proof}
For (a) let $1+\alp\in V_{n,R}$ and $1+\beta\in V_{k,R}$.
Then
\[
\begin{split}
[1+\alp,1+\beta]&=(1-\alp+\alp^2-+\cdots)(1-\beta+\beta^2-+\cdots)(1+\alp)(1+\beta) \\
&\in 1+R\langle\langle X_A\rangle\rangle\alp\beta R\langle\langle X_A\rangle\rangle
   +R\langle\langle X_A\rangle\rangle\beta\alp R\langle\langle X_A\rangle\rangle\leq V_{n+k,R}.
\end{split}
\]

For (b), use the (non-commutative) binomial formula.
\end{proof}

\section{Free profinite groups}
\label{section on free profinite groups}
We recall from \cite{FriedJarden08}*{\S17.4} the following terminology and facts on free profinite groups.

Let $G$ be a profinite group and $A$ a set.
A map $\varphi\colon A\to G$  \textbf{converges to $1$} if for every open normal subgroup $M$ of $G$,
the set $A\setminus\varphi\inv(M)$ is finite.

Let $S$ be a profinite group.
We say that $S$ is a \textbf{free profinite group on basis $A$} with respect to a map $\iota\colon A\to S$
if
\begin{enumerate}
\item[(i)]
$\iota\colon A\to S$ converges to $1$ and $\iota(A)$ generates $S$;
\item[(ii)]
For every profinite group $G$
and a continuous map $\varphi\colon A\to G$ converging to $1$, there
is a unique continuous homomorphism $\hat\varphi\colon S\to G$ with
$\varphi=\hat\varphi\circ\iota$ on $A$.
\end{enumerate}
A free profinite group on $A$ exists, and is unique up to a continuous isomorphism.
We denote it  by  $S_A$.
Necessarily, $\iota$ is injective, and we identify $A$ with its image in $S_A$.
The group $S_A$ is projective, whence has cohomological dimension $\leq1$
 \cite{NeukirchSchmidtWingberg}*{Cor.\ 3.5.16}.

Now let $R$ be a profinite unital ring.
The map $\sum_wc_wX_w\mapsto(c_w)_w$ identifies $R\langle\langle X_A\rangle\rangle$ with $R^{A^*}$.
This induces on the additive group of $R\langle\langle X_A\rangle\rangle$ a profinite topology.
Moreover, the multiplication map in  $R\langle\langle X_A\rangle\rangle$ is continuous,
making it a profinite topological ring.

Now take $A$ finite.
The ring  $R\langle\langle X_A\rangle\rangle$ is profinite ring, so
$R\langle\langle X_A\rangle\rangle^\times$ is a profinite group.
Hence the map $a\mapsto 1+X_a$, $a\in A$, extends to the (continuous)   \textbf{Magnus homomorphism}
$\Lam_{S_A,R}\colon S_A\to R\langle\langle X_A\rangle\rangle^\times$.

This generalizes to arbitrary $A$ as follows:
Let $B$ be a finite subset of $A$.
The map $\varphi\colon A\to S_B$,  given by $a\mapsto a$ for $a\in B$,
and $a\mapsto 1$ for $a\in A\setminus B$, converges to $1$.
It extends to a unique continuous epimorphism $S_A\to S_B$.
Also, there is a continuous $R$-algebra epimorphism
$R\langle\langle X_A\rangle\rangle\to R\langle\langle X_B\rangle\rangle$,
given by $X_b\mapsto X_b$ for $b\in B$ and $X_a\mapsto0$ for $a\in A\setminus B$.
Then
\[
S_A=\invlim S_B,  \
R\langle\langle X_A\rangle\rangle=\invlim R\langle\langle X_B\rangle\rangle, \
R\langle\langle X_A\rangle\rangle^\times=\invlim R\langle\langle X_B\rangle\rangle^\times,
\]
where $B$ ranges over all finite subsets of $A$.
We now define $\Lam_{S_A,R}\colon S_A\to R\langle\langle X_A\rangle\rangle^\times$
to be the inverse limit of the maps $\Lam_{S_B,R}$.
Thus  $\Lam_{S_A,R}(a)=1+X_a$ for $a\in A$.

From now on we abbreviate $S=S_A$.
For $\sig\in S$ we set
\[
\Lam_{S,R}(\sig)=\sum_{w\in A^*}\eps_{w,R}(\sig)X_w
\]
with $\eps_{w,R}(\sig)\in R$.
Thus $\eps_{\emptyset,R}(\sig)=1$.

For a positive integer $n$ let $S_{n,R}=\Lam_{S,R}\inv(V_{n,R})$.
It is closed subgroup of $S$,  and $S_{1,R}=S$.
Also, $S_{n,R}=\invlim (S_B)_{n,R}$, with $B$ ranging over all finite subsets of $A$.
The next lemma records a few well-known facts.

\begin{lem}
\label{ttt}
Let  $w\in A^*$.
\begin{enumerate}
\item[(a)]
For $\sig,\tau\in S$ one has
$\eps_{w,R}(\sig\tau)=\sum\eps_{w_1,R}(\sig)\eps_{w_2,R}(\tau)$,
where the sum is over all $w_1,w_2\in A^*$ with $w=w_1w_2$.
\item[(b)]
For $n=|w|$, the restriction $\eps_{w,R}\colon S_{n,R}\to R$ is a homomorphism.
\item[(c)]
For $a\in A$,  the map $\eps_{(a),R}\colon S\to R$ is a homomorphism.
\item[(d)]
$\eps_{(a),R}(a)=1$ for $a\in A$ and $\eps_{(a),R}(a')=0$ for $a,a'\in A$ distinct.
\end{enumerate}
\end{lem}
\begin{proof}
(a) is a restatement of  $\Lam_{S,R}(\sig\tau)=\Lam_{S,R}(\sig)\Lam_{S,R}(\tau)$.
(b) follows from (a), and (c) is a special case of (b).
(d) is immediate from the definition.
\end{proof}

\section{The filtration $S_{n,\dbZ/m}$}
\label{section of the filtration S_{n,Z/m}}
Let $S=S_A$ be as before, and assume that  $R=\dbZ/m$ for an integer $m\geq2$.
Lemma \ref{filtration of power series} implies that
\begin{equation}
\label{inclusion of S_n}
S_{n,\dbZ/m}^m[S,S_{n,\dbZ/m}]\leq S_{n+1,\dbZ/m}.
\end{equation}

For $n=1$, this is an equality:

\begin{lem}
\label{description of S2}
$S^m[S,S]=S_{2,\dbZ/m}$.
\end{lem}
\begin{proof}
By (\ref{inclusion of S_n}), $S^m[S,S]\leq S_{2,\dbZ/m}$.

For the converse, we may use an inverse limit argument to assume that $A$ is finite.
Let  $\bar S=S/S^m[S,S]$ and let $\bar\Lam_{S,\dbZ/m}\colon \bar S\to V_{1,\dbZ/m}/V_{2,\dbZ/m}$
be the homomorphism induced by $\Lam_{S,\dbZ/m}$.
Use Lemma \ref{filtration of power series} to obtain that $ V_{1,\dbZ/m}/V_{2,\dbZ/m}$
is a free $\dbZ/m$-module on generators $(1+X_a)V_{2,\dbZ/m}$, $a\in A$.
Since $\bar S$ is abelian of exponent $m$, we may define a homomorphism $\lam\colon V_{1,\dbZ/m}/V_{2,\dbZ/m}\to\bar S$ by mapping this generator to the image $\bar a$ of $a$ in $\bar S$.
Then $(\lam\circ\bar\Lam_{S,\dbZ/m})(\bar a)=\bar a$ for $a\in A$, implying that $\lam\circ\bar\Lam_{S,\dbZ/m}=\id_{\bar S}$ and $\bar\Lam_{S,\dbZ/m}$ is injective.
But $S_{2,\dbZ/m}/S^m[S,S]$ is mapped trivially by $\bar\Lam_{S,\dbZ/m}$, and therefore is trivial.
\end{proof}

Now let $m=p$ prime.
For a profinite group $G$ let  $I_{\dbF_p}(G)$ be the augmentation ideal in the complete group ring
$\dbF_p[[G]]$, i.e., the closed ideal generated by all elements $g-1$, with $g\in G$.
When $G$ is finite (whence discrete), a theorem of Jennings and Brauer (\cite{Jennings41}*{Th.\ 5.5},
\cite{DixonDuSautoyMannSegal99}*{Th.\ 12.9}) identifies $G\cap(1+I_{\dbF_p}(G)^n)$ with the $n$th term $G_n$
in the Zassenhaus filtration of $G$, $n\geq1$.
An inverse limit argument extends this to arbitrary profinite groups $G$; see \cite{Lazard54}*{Cor.\ 6.10} and \cite{Quillen68}*{Th.\ 2.4}.

The map $1+\sum_{|w|\geq n}c_wX_w \mapsto (c_w)_{|w|=n}$ induces
a group isomorphism  $V_{n,\dbF_p}/V_{n+1,\dbF_p}\xrightarrow{\sim}\prod_{w\in A^*, |w|=n}\dbF_p$.
When $A$ is finite, $V_{1,\dbF_p}/V_{n,\dbF_p}$ is therefore a finite $p$-group for every $n$, so
$V_{1,\dbF_p}=\invlim V_{1,\dbF_p}/V_{n,\dbF_p}$ is a pro-$p$ group.
We have the following power series variant of the above results of Jennings, Lazard and Quillen:

\begin{prop}
\label{Jennings-Brauer}
$S_n=S_{n,\dbF_p}$.
\end{prop}
\begin{proof}
By an inverse limit argument, we may assume that the basis $A$ is finite.
Let $\bar S$ be the maximal pro-$p$ quotient of $S$.
It is a free pro-$p$ group.
By the definition of the Zassenhaus filtration, $S/S_n$ has a $p$-power exponent.
Hence the preimage of $\bar S_n$ under the epimorphism $S\to\bar S$ is $S_n$.

Also, $\Lam_{S,\dbF_p}$ breaks via a homomorphism $\bar\Lam_{S,\dbF_p}\colon\bar S\to V_{1,\dbF_p}$.
It extends to a continuous $\dbF_p$-algebra homomorphism
 $\bar\Lam_{\bar S,\dbF_p}\colon \dbF_p[[\bar S]]\to \dbF_p\langle\langle X_A\rangle\rangle$,
 which by \cite{Koch02}*{Th.\ 7.16},  is an isomorphism.
Let $J=V_{1,\dbF_p}-1$ be the ideal of all power series in $\dbF_p\langle\langle X_A\rangle\rangle$
with constant term $0$.
Using the identity $\sig\tau-1=(\sig-1)(\tau-1)+(\sig-1)+(\tau-1)$ for $\sig,\tau\in \bar S$ we see that
$\bar\Lam_{\bar S,\dbF_p}$  maps $I_{\dbF_p}(\bar S)$ onto $J$.
Therefore it maps $1+I_{\dbF_p}(\bar S)^n$ bijectively onto $1+J^n=V_{n,\dbF_p}$.
By the Jennings--Brauer theorem, $\bar S_n$ is therefore the preimage of $V_{n,\dbF_p}$ in $\bar S$
under $\bar\Lam_{\bar S,\dbF_p}$.
We conclude that $S_n$ is the preimage of $V_{n,\dbF_p}$ in $S$
under $\Lam_{S,\dbF_p}$, i.e.,  $S_n=S_{n,\dbF_p}$.
\end{proof}

\begin{lem}
\label{aux hom}
Let $c_1\nek c_n\colon S=S_A\to \dbF_p$ be continuous homomorphisms and
$B=\{b_1\nek b_n\}$  be a set of size $n$.
There is a continuous homomorphism $\phi\colon S_A\to S_B$ such that
$\eps^B_{(b_i),\dbF_p}\circ\phi=c_i$, $i=1,2\nek n$, where $\eps^B$ denotes the Magnus
coefficient with respect to $S_B$.
\end{lem}
\begin{proof}
An inverse limit argument reduces this to the case where $A$ is finite.
For every $a\in A$ and $1\leq i\leq n$ choose $\hat c_i(a)\in \dbZ$ with $c_i(a)=\hat c_i(a)\pmod{p\dbZ}$.
The map $A\to S_B$, $a\mapsto \prod_{i=1}^nb_i^{\hat c_i(a)}$,
extends to a continuous homomorphism $\phi\colon S_A\to S_B$.
Then $\eps^B_{(b_i),\dbF_p}(\phi(a))=c_i(a)$ for every $a\in A$ and $1\leq i\leq n$, and the assertion follows.
\end{proof}

\section{Unipotent Matrices}
Consider integers $n\geq2$ and $d\geq0$, and let $R$ be a (discrete) finite ring.
Let $T_{n,d}(R)$ be the set of all  $n\times n$ matrices $(a_{ij})$ over $R$ with the $(1,n)$
entry omitted and such that $a_{ij}=0$ for $j-i\leq d-1$ (in particular, $(a_{ij})$ is upper-triangular).
It is an $R$-algebra with respect to the standard operations.
Note that  $T_{n,d}T_{n,d'}\subseteq T_{n,d+d'}$.
Furthermore,  for every entry $(i,j)\neq(1,n)$ we have $j-i\leq n-2$.
Hence $T_{n,d}=\{0\}$ for $n-1\leq d$.
We denote the $n\times n$ identity matrix with the $(1,n)$ entry omitted by $I_n$.

Let $\U_n(R)$ be the group of all upper-triangular unipotent $n\times n$ matrices over $R$.
Let $\bar \U_n(R)=I_n+T_{n,1}(R)$ be the group of all unipotent (punctured) matrices in $T_{n,0}$.
Let $\pi\colon \U_n(R)\to\bar \U_n(R)$ be the obvious forgetful epimorphism.
Its kernel consists of all matrices in $\U_n(R)$ which are zero
except for the main diagonal and at the entry $(1,n)$, and is therefore isomorphic to the additive group of $R$.
We obtain a central extension of groups
\[
0\to R\to \U_n(R)\xrightarrow{\pi} \bar \U_n(R)\to1.
\]
We endow $\U_n(R)$, $\bar \U_n(R)$ with the discrete topologies.

For the rest of this section we set $S=S_A$ and $R=\dbZ/m$, with $m\geq2$.

\begin{prop}
\label{factoring lemma}
Every continuous homomorphism $\gam\colon S\to\bar \U_n(R)$ is trivial on $S_{n-1,R}$.
\end{prop}
\begin{proof}
An inverse limit argument reduces this to the case where $A$ is finite.

For $w=(a_1\nek a_d)\in A^*$ we  set $M_w=\prod_{j=1}^d(\gam(a_j)-I_n)\in T_{n,d}$ (where $M_\emptyset=I_n$).
Since $T_{n,d}=\{0\}$ for $n-1\leq d$, we may therefore define a unital $R$-algebra homomorphism
$h\colon R\langle\langle X_A\rangle\rangle\to T_{n,0}(R)$ by
\[
h\bigl(\sum_{w\in A^*}c_wX_w\bigr)=\sum_{{w\in A^*}\atop{|w|\leq n-2}}c_wM_w.
\]
Note that $h$ is continuous with respect to the profinite topology on $R\langle\langle X_A\rangle\rangle$
and the discrete
topology on  $T_{n,0}(R)$.
The restriction of $h$ to $V_{1,R}$ is a continuous group homomorphism $h\colon V_{1,R}\to\bar \U_n(R)\subseteq T_{n,0}(R)$.
One has $\gam=h\circ\Lam_{S,R}$ on $A$, whence on $S$.
For $\sig\in S_{n-1,R}$ this gives $\gam(\sig)=(h\circ\Lam_{S,R})(\sig)=I_n$.
\end{proof}

\begin{cor}
Every continuous  homomorphism $\gam\colon S\to \U_n(R)$ is trivial on $S_{n,R}$.
\end{cor}
\begin{proof}
Embed $\U_n(R)$ in $\bar \U_{n+1}(R)$ and use Proposition \ref{factoring lemma}.
\end{proof}

\begin{rem}
\rm
In the analogous case of coefficient ring $\dbZ$, the last two results follow from \cite{MagnusKarrassSolitar}*{\S5.5, Cor.\ 5.7} and the fact that $\bar\U_n(\dbZ)$, $\U_n(\dbZ)$ are nilpotent of the appropriate degrees.
\end{rem}

Given a continuous homomorphism $\bar\gam\colon S/S_{n-1,R}\to \bar \U_n(R)$ we write $\U_n(R)\times_{\bar \U_n(R)}(S/S_{n-1,R})$
for the fiber product with respect to $\pi$ and $\bar\gam$.

\begin{lem}
\label{kernels}
Let $N$ and $M_0$ be closed normal subgroups of $S$ such that $N\leq S_{n-1,R}\cap M_0$.
The following conditions are equivalent:
\begin{enumerate}
\item[(1)]
there exist a continuous homomorphism $\bar\gam\colon S/S_{n-1,R}\to \bar \U_n(R)$ and a
 continuous homomorphism
\[
\hat\Phi\colon S/N\to \U_n(R)\times_{\bar \U_n(R)}(S/S_{n-1,R})
\]
which commutes with the projections to $S/S_{n-1,R}$, and such that $M_0/N=\Ker(\hat\Phi)$.
\item[(2)]
there exists a continuous homomorphism $\Phi\colon S/N\to \U_n(R)$ such that
$M_0/N=\Ker(\Phi)\cap (S_{n-1,R}/N)$;
\item[(3)]
there is a closed normal subgroup $M$ of $S$ containing $N$ such that $S/M$
embeds in $\U_n(R)$ and $M_0=M\cap S_{n-1,R}$.
\end{enumerate}
\end{lem}
\begin{proof}
(1)$\Rightarrow$(2): \quad
For $i=1,2$ let $\pr_i$ be  the projection on the $i$-th coordinate of the fiber product,
and $\pr\colon S\to S/N$ the natural map.
For $\hat\Phi$ as in (1), we set $\Phi=\pr_1\circ\hat\Phi$.
We get a commutative diagram
\[
\xymatrix{
S\ar[r]^{\pr}& S/N\ar[rd]^{\hat\Phi}\ar@/_1pc/[rdd]_(0.65){\Phi}\ar@/^/[rrd] &&  \\
& & \U_n(R)\times_{\bar \U_n(R)}(S/S_{n-1,R}) \ar[r]_{\qquad\qquad\pr_2}\ar[d]^{\pr_1}
& S/S_{n-1,R}\ar[d]^{\bar\gam} \\
&& \U_n(R) \ar[r]^{\pi} & \bar \U_n(R).
}
\]
Further,  $M_0/N=\Ker(\hat\Phi)=\Ker(\Phi)\cap (S_{n-1,R}/N)$.

\medskip

(2)$\Rightarrow$(1): \quad
Given $\Phi$ as in (2), the homomorphism $\pi\circ\Phi\circ\pr\colon S\to\bar \U_n(R)$
factors via a continuous
homomorphism $\bar\gam\colon S/S_{n-1,R}\to\bar \U_n(R)$, by Proposition \ref{factoring lemma}.
Thus the outer part of the diagram above commutes.
The universal property of the fiber product yields a continuous homomorphism $\hat\Phi$
making the two triangles commutative.
We have
\[
M_0/N=\Ker(\Phi)\cap (S_{n-1,R}/N)=\Ker(\hat\Phi).
\]

(2)$\Leftrightarrow$(3):  \quad
Take $\Ker(\Phi)=M/N$.
\end{proof}

One has the following important connection between homomorphisms as discussed above and words:

\begin{lem}
\label{uuu}
Let $w=(a_1\nek a_n)\in A^*$.
Define maps $\gam_1,\gam_2\colon S\to \U_{n+1}(R)$ by
\[
\gam_1(\sig)_{ij}=\eps_{(a_i\nek a_{j-1}),R}(\sig), \quad
\gam_2(\sig)_{ij}=(-1)^{j-i}\eps_{(a_i\nek a_{j-1}),R}(\sig)
\]
for $\sig\in S$ and $i<j$ (the other entries being obvious).
Then $\gam_1,\gam_2$ are continuous group homomorphisms.
\end{lem}
\begin{proof}
For $\gam_1$ use Lemma \ref{ttt}(a).
For $\gam_2$ observe that the map $(a_{ij})\mapsto((-1)^{j-i}a_{ij})$
is an automorphism of $\U_{n+1}(R)$, and compose it with $\gam_1$.
\end{proof}

\section{Massey products for inhomogenous cochains}
\label{section on Massey products for inhomogenous cochains}
Let $G$ be a profinite group which acts trivially and continuously on the unital finite (discrete) ring $R$.
The complex  $(C^\bullet(G,R),\partial)$ of continuous inhomogenous $G$-cochains into the additive group of
$R$, and $C^r(G,R)=0$ for $r<0$, is a DGA with the cup product \cite{NeukirchSchmidtWingberg}*{Ch.\ I, \S2}.
We recall  that for $c\in C^1(G,R)$ and $\sig,\tau\in G$ one has $(\partial c)(\sig,\tau)=c(\sig)+c(\tau)-c(\sig\tau)$.
Thus  the $1$-cocycles are the continuous homomorphisms $c\colon G\to R$.
We now focus on defining systems in $C^\bullet(G,R)$.

As observed by  Dwyer \cite{Dwyer75}*{\S2} in the discrete context, one may view defining systems of size $n$
in $C^\bullet(G,R)$ as continuous homomorphisms $G\to\bar \U_{n+1}(R)$, as follows.
Define a bijection between the systems of $1$-cochains $c_{ij}\in C^1(G,R)$, $1\leq i\leq j\leq n$, $(i,j)\neq(1,n)$,
and the continuous maps $\bar\gam\colon G\to \bar \U_{n+1}(R)$ by
\[
\bar\gam(\sig)_{ij}=
(-1)^{j-i}c_{i,j-1}(\sig)
\]
for $\sig\in G$ and $1\leq i<j\leq n+1$, $(i,j)\neq(1,n+1)$ (where the other entries are obvious).
Under this bijection one has for $\sig,\tau\in G$,
\begin{equation}
\label{hjkl}
\widetilde c_{il}(\sig,\tau)=-\sum_{r=i}^{l-1}c_{ir}(\sig)c_{r+1,l}(\tau)
=(-1)^{l-i}\sum_{k=i+1}^{l}\bar\gam(\sig)_{ik}\bar\gam(\tau)_{k,l+1}.
\end{equation}

\begin{lem}
\label{homs and defining systems}
 $\bar\gam$ is a homomorphism if and only if $(c_{ij})$ is a defining system of size $n$ for  $C^\bullet(G,R)$.
\end{lem}
\begin{proof}
The map $\bar\gam$ is a homomorphism if and only if for every $\sig,\tau\in G$ and $1\leq i\leq l\leq n$ with $(i,l)\neq(1,n)$,
\[
\bar\gam(\sig\tau)_{i,l+1}=\bar\gam(\tau)_{i,l+1}+\sum_{k=i+1}^{l}\bar\gam(\sig)_{ik}\bar\gam(\tau)_{k,l+1}+\bar\gam(\sig)_{i,l+1}.
\]
By  (\ref{hjkl}), this means that
$c_{il}(\sig\tau)=c_{il}(\tau)-\widetilde c_{il}(\sig,\tau)+c_{il}(\sig)$.
Equivalently,  $\widetilde c_{il}(\sig,\tau)=(\partial c_{il})(\sig,\tau)$, i.e., $(c_{ij})$ is a defining system.
\end{proof}

\begin{rem}
\label{homs and defining systems 2}
\rm
The same formula gives a bijection between the systems $c_{ij}\in C^1(G,R)$, $1\leq i\leq j\leq n$, and the continuous maps
$\gam\colon G\to \U_{n+1}(R)$.
Moreover, $\gam$ is a homomorphism if and only if  $(c_{ij})$ is a defining system such that in addition
$\widetilde c_{1n}=\partial c_{1n}$.
\end{rem}

The following fact (with different sign conventions) is stated without a proof in \cite{Dwyer75}*{p.\ 182, Remark}; see also \cite{Wickelgren12}*{\S2.4}.

\begin{prop}
\label{upper triangular matrices}
Let $\bar\gam\colon G\to\bar \U_{n+1}(R)$ correspond to a defining system $(c_{ij})$ as above.
The central extension associated with $(-1)^{n-1}\widetilde c_{1n}$ is
\[
0\to R\to \U_{n+1}(R)\times_{\bar \U_{n+1}(R)}G\to G\to1,
\]
where  the fiber product is with respect to $\pi$ and $\bar\gam$.
\end{prop}
\begin{proof}
Since $(-1)^{n-1}\widetilde c_{1n}$ is a $2$-cocycle,  $B=R\times G$ is a group with respect to the product
$(r,\sig)*(s,\tau)=(r+s+(-1)^{n-1}\widetilde c_{1n}(\sig,\tau),\sig\tau)$.
Then the central extension corresponding to $(-1)^{n-1}\widetilde c_{1n}$ is \cite{NeukirchSchmidtWingberg}*{Th.\ 1.2.4}
\[
0\to R\to B\to G\to 1.
\]

The map $h\colon \U_{n+1}(R)\times_{\bar \U_{n+1}(R)}G\to B$, $((a_{ij}),\sig)\mapsto (a_{1,n+1},\sig)$, is
a bijection commuting with the projections to $G$.
To show that $h$ is a homomorphism, take $((a_{ij}),\sig),((b_{ij}),\tau)\in \U_{n+1}(R)\times_{\bar \U_{n+1}(R)}G$.
Thus
\[
a_{ij}=\bar\gam(\sig)_{ij}=(-1)^{j-i}c_{i,j-1}(\sig), \qquad
b_{ij}=\bar\gam(\tau)_{ij}=(-1)^{j-i}c_{i,j-1}(\tau)
\]
for $1\leq i<j\leq n+1$, $(i,j)\neq(1,n+1)$.
By (\ref{hjkl}),
$\sum_{k=2}^na_{1k}b_{k,n+1}=(-1)^{n-1}\widetilde c_{1n}(\sig,\tau)$.
Hence
\[
\begin{split}
&h\Bigl(((a_{ij}),\sig)((b_{ij}),\tau)\Bigr)=\Bigl(\sum_{k=1}^{n+1}a_{1k}b_{k,n+1},\sig\tau\Bigr)\\
&=\Bigl(a_{1,n+1}+b_{1,n+1}+(-1)^{n-1}\widetilde c_{1n}(\sig,\tau),\sig\tau\Bigr)
=h\Bigl(((a_{ij}),\sig)\Bigr)*h\Bigl(((b_{ij}),\tau)\Bigr).
\end{split}
\]
\end{proof}

For the rest of the paper we set again  $R=\dbZ/m$.
Let $S=S_A$ and let $N$ be a normal closed subgroup of  $S$ contained in $S_{n,\dbZ/m}$, with $n\geq2$.
Let $\inf_S\colon Z^1(S/N,\dbZ/m)\to Z^1(S,\dbZ/m)$ be the inflation map on $1$-cocycles (i.e., continuous homomorphisms).

\begin{prop}
\label{extension property}
Given continuous homomorphisms $c_1\nek c_n\colon S\to \dbZ/m$, there is a defining system
$(\bar c_{ij})$ of size $n$ in $C^\bullet(S/N,\dbZ/m)$ with $c_i=\inf_S(\bar c_{ii})$, $i=1,2\nek n$.
\end{prop}
\begin{proof}
Let $S_B$ be a free profinite group on a set $B=\{b_1,b_2\nek b_n\}$ of  $n$ elements.
We write $\eps^B_{w,\dbZ/m}$ for the corresponding Magnus coefficients.
Lemma \ref{aux hom} yields  a continuous homomorphism $\phi\colon S\to S_B$ such that
$\eps^B_{(b_i),\dbZ/m}\circ\phi=c_i$, $i=1,2\nek n$.
We define a map $\gam'\colon S_B\to \U_{n+1}(\dbZ/m)$ by
$\gam'(\sig')_{ij}=(-1)^{j-i}\eps'_{(b_i\nek b_{j-1}),\dbZ/m}(\sig')$ for $i<j$ and $\sig'\in S_B$.
By Lemma \ref{uuu}, $\gam'$ is a continuous homomorphism.
The composition  $\gam=\pi\circ\gam'\circ\phi\colon S\to \bar \U_{n+1}(\dbZ/m)$ is also a continuous homomorphism.
By Proposition \ref{factoring lemma}, it factors via a continuous homomorphism
$\bar\gam\colon S/N\to\bar \U_{n+1}(\dbZ/m)$:
\[
\xymatrix{
S\ar[r]^{\phi}\ar[d]\ar[rrd]^{\gam} & S_B\ar[r]^{\gam'\qquad} & \U_{n+1}(\dbZ/m)\ar[d]_{\pi} \\
S/N\ar[rr]^{\bar \gam} && \bar \U_{n+1}(\dbZ/m).
}
\]
Let $(\bar c_{ij})$ be the defining system of size $n$ on  $C^\bullet(S/N,\dbZ/m)$ associated with $\bar \gam$,
in the sense of Lemma \ref{homs and defining systems}.
Then for $\sig\in S$ and $1\leq i\leq n$,
\[
\textstyle
(\inf_S(\bar c_{ii}))(\sig)=-\gam_{i,i+1}(\sig)=-\gam'_{i,i+1}(\phi(\sig))
=\eps^B_{(b_i),\dbZ/m}(\phi(\sig))=c_i(\sig).
\qedhere
\]
\end{proof}

For a profinite group $G$ acting trivially on $\dbZ/m$,
let $H^i(G)=H^i(G,\dbZ/m)$ be the $i$-th cohomology group
corresponding to the DGA $C^\bullet(G,\dbZ/m)$ over $\dbZ/m$.
In view of Proposition \ref{extension property}, the assumption of Proposition \ref{uniqueness}
is satisfied for $C^\bullet(S/N,\dbZ/m)$.
Consequently, as explained in \S\ref{section on Massey products}, there is a well-defined Massey product
\[
\langle\cdot\nek\cdot\rangle\colon H^1(S/N)^n\to H^2(S/N).
\]
This was earlier shown using a different method by Vogel \cite{Vogel05}*{Th.\ A3} for $m=p$ prime.
We write $H^2(S/N)_{n-\Massey}$ for the image of this map
(in \cite{Dwyer75}*{\S3} this functor is denoted by $\Phi^n_{\dbF_p}$).

\begin{exam}
\label{cup products}
\rm
For $n=2$ and $\chi_1,\chi_2\in H^1(S/N)$ we have by construction
$\langle\chi_1,\chi_2\rangle=-\chi_1\cup\chi_2\in H^2(S/N)$, where $\cup$ denotes the cup product.
Thus in the terminology of \cite{CheboluEfratMinac12}, $H^2(S/N)_{2-\Massey}=H^2(S/N)_\dec$.
\end{exam}

For $a\in A$ we may consider  $\eps_{(a),\dbZ/m}\in H^1(S)$ also as an element
of $H^1(S/N)$ (see Lemma \ref{description of S2}).
With this convention we have

\begin{lem}
\label{gens for H2 Massey}
The Massey products
$\psi_w=\langle\eps_{(a_1),\dbZ/m}\nek\eps_{(a_n),\dbZ/m}\rangle$,
where  $w=(a_1\nek a_n)\in A^*$, generate $H^2(S/N)_{n-\Massey}$.
\end{lem}
\begin{proof}
In view of Lemma \ref{ttt}(d), $\eps_{(a),\dbZ/m}$, where $a\in A$, generate $H^1(S/N)=\Hom(S/N,\dbZ/m)$.
Now use the multi-linearity of the Massey product.
\end{proof}

\section{Cohomological duality}
Let $G$ be a profinite group acting trivially on $\dbZ/m$ and let $N$ be a closed normal subgroup of $G$.
One has the 5-term exact sequence for the cohomology groups with coefficients in $\dbZ/m$
\cite{NeukirchSchmidtWingberg}*{Prop.\ 1.6.7}:
\[
0\to H^1(G/N)\xrightarrow{\inf} H^1(G)\xrightarrow{\res} H^1(N)^G\xrightarrow{\trg}
H^2(G/N)\xrightarrow{\inf} H^2(G).
\]
When $N\leq G^m[G,G]$, the inflation map $\inf\colon H^1(G/N)\to H^1(G)$ is surjective,
so the transgression map  $\trg$ identifies $H^1(N)^G$ with $\Ker(H^2(G/N)\xrightarrow{\inf} H^2(G))$.
Therefore \cite{EfratMinac11a}*{Cor. 2.2} gives a non-degenerate bilinear map
\begin{equation}
\label{abc}
\begin{split}
(\cdot,\cdot)'\colon N/N^m&[G,N]\times \Ker(H^2(G/N)\xrightarrow{\inf} H^2(G))\to\dbZ/m, \quad \\
&(\sig N^m[G,N],\alp)'=(\trg\inv(\alp))(\sig).
\end{split}
\end{equation}

We will need the following result from \cite{EfratMinac11b}*{Prop.\ 3.2}.
Note that while it is stated for $m$ a prime power, this is not needed in its proof.

\begin{prop}
\label{EM2 Prop 32}
Let $T,T_0$ be closed normal subgroups of $G$ such that $T^m[G,T]\leq T_0\leq T\leq G^m[G,G]$.
Let $H$ be a subgroup of $H^2(G/T)$ and $H_0$ a set of generators of
$H\cap\trg(H^1(T)^G)=\Ker(\inf\colon H\to H^2(G))$.
The following conditions are equivalent:
\begin{enumerate}
\item[(a)]
$(\cdot,\cdot)'$ (for $N=T$) induces a non-degenerate bilinear map
\[
T/T_0\times \Ker(H\xrightarrow{\inf}H^2(G))\to\dbZ/m;
\]
\item[(b)]
there is an exact sequence
\[
0\to \Ker(H^2(G/T)\xrightarrow{\inf} H^2(G/T_0))\to H\xrightarrow{\inf} H^2(G);
\]
\item[(c)]
$T_0=\bigcap\Ker(\Psi)$, with $\Psi$ ranging over all homomorphisms with a commutative diagram
\[
\xymatrix{
&&&&G\ar[dl]_{\Psi}\ar[d]&\\
\omega: & 0\ar[r]&\dbZ/m\ar[r] & C\ar[r] & G/T\ar[r] & 1,
}
\]
where $\omega$ is a central extension associated with some element of $H_0$.
\end{enumerate}
\end{prop}

We now restrict ourselves to the case where $G$ is a free profinite group $S=S_A$.
Let $N$ be a normal subgroup of $S$ contained in $S_{n,\dbZ/m}$, where $n\geq2$.
Then $N\leq S_{2,\dbZ/m}=S^m[S,S]$, by Lemma \ref{description of S2}.
As $H^2(S)=0$, the map (\ref{abc}) then becomes a non-degenerate bilinear map
\begin{equation}
\label{the pairing (.)'}
(\cdot,\cdot)'\colon N/N^m[S,N]\times H^2(S/N)\to\dbZ/m.
\end{equation}

Variants of the following fundamental fact were proved by Dwyer \cite{Dwyer75}*{Prop.\ 4.1},
Fenn and Sjerve \cite{FennSjerve84}*{Th.\ 6.6}, Morishita \cite{Morishita04}*{Cor.\ 2.2.3},  Vogel \cite{Vogel05}*{Th.\ A3}
and Wickelgren \cite{Wickelgren09}*{Prop.\ 2.3.7}.
We prove it here in our terminology and setup.
Let $\psi_w$ be as in Lemma \ref{gens for H2 Massey}.

\begin{thm}
\label{fundamental duality}
For $\sig\in N$ and $w\in A^*$ one has $(\bar\sig,\psi_w)'=\eps_{w,\dbZ/m}(\sig)$.
\end{thm}
\begin{proof}
Let $w=(a_1\nek a_n)$.
Lemma \ref{uuu} gives a continuous homomorphism $\gam\colon S\to \U_{n+1}(\dbZ/m)$, where
$\gam(\sig)_{ij}=(-1)^{j-i}\eps_{(a_i\nek a_{j-1}),\dbZ/m}(\sig)$ for $\sig\in S$, $i<j$.
By Proposition \ref{factoring lemma}, it induces a continuous homomorphism
$\bar\gam\colon S/N\to\bar \U_{n+1}(\dbZ/m)$ such that $\bar\gam\circ\lam=\pi\circ\gam$, where $\lam\colon S\to S/N$ and
$\pi\colon \U_{n+1}(\dbZ/m)\to\bar \U_{n+1}(\dbZ/m)$ are the natural epimorphisms.

Let $(c_{ij})$ and $(\bar c_{ij})$ correspond to $\gam$ and $\bar\gam$, respectively, under the bijections of
Lemma \ref{homs and defining systems} and Remark \ref{homs and defining systems 2}.
Thus
\[
c_{ij}\in C^1(S,\dbZ/m), \quad \bar c_{ij}\in C^1(S/N,\dbZ/m), \quad  c_{ij}=\textstyle\inf_S(\bar c_{ij}),
\]
and  $(\bar c_{ij})$ is a defining system of size $n$ in $C^\bullet(S/N,\dbZ/m)$.
Furthermore, $\widetilde c_{1n}=\partial c_{1n}$.
By construction, $\eps_{(a_i),\dbZ/m}=\bar c_{ii}$ as elements of $H^1(S/N)$, $i=1,2\nek n$,
and $\eps_{w,\dbZ/m}=(-1)^n\gam_{1,n+1}=c_{1n}$ in $H^1(S)$.

Now by the definition of the transgression \cite{NeukirchSchmidtWingberg}*{Prop.\ 1.6.5},
$[\widetilde{\bar c_{1n}}]=\trg[c_{1n}|_N]$.
Altogether,
\[
\psi_w=\langle\eps_{(a_1),\dbZ/m}\nek\eps_{(a_n),\dbZ/m}\rangle
=\langle\bar c_{11}\nek\bar c_{nn}\rangle
=[\widetilde{\bar c_{1n}}]=\trg[c_{1n}|_N].
\]
Therefore
$(\bar\sig,\psi_w)'=c_{1n}(\sig)=\eps_{w,\dbZ/m}(\sig)$.
\end{proof}

\section{Proof of Theorem B}
\label{section on Proof of Theorem B}
As before, let $S=S_A$ and let $N$ be a closed normal subgroup of $S$ with $N\leq S_{n,\dbZ/m}$,  $n\geq2$.
By Lemma \ref{description of S2}, $N\leq S_{2,\dbZ/m}=S^m[S,S]$.

\begin{thm}
\label{generalized Thm B ver1}
\begin{enumerate}
\item[(a)]
$(\cdot,\cdot)'$ induces a perfect pairing
\[
NS_{n+1,\dbZ/m}/S_{n+1,\dbZ/m}\times H^2(S/N)_{n-\Massey}\to \dbZ/m.
\]
\item[(b)]
There is a natural exact sequence
\[
0\to H^2(S/N)_{n-\Massey}\hookrightarrow  H^2(S/N)\xrightarrow{\inf} H^2(S/(N\cap S_{n+1,\dbZ/m})).
\]
\end{enumerate}
\end{thm}
\begin{proof}
(a) \quad
In view of Lemma \ref{ttt}(b), there is a $\dbZ/m$-bilinear map
\[
S_{n,\dbZ/m} \times \bigoplus_{{w\in A^*}\atop{|w|=n}}\dbZ/m \to \dbZ/m, \qquad
\Bigl(\sig,(\bar r_w)\Bigr)=\sum_{|w|=n}\bar r_w \eps_{w,\dbZ/m}(\sig)
\]
with left kernel $S_{n+1,\dbZ/m}$.
Let $\psi_w$ be as in Lemma \ref{gens for H2 Massey}.
We consider the diagram of bilinear maps
\begin{equation}
\label{main cd}
\xymatrix{
S_{n,\dbZ/m}/S_{n+1,\dbZ/m} & *-<3pc>{\times} & \strut\bigoplus_{w\in A^*, |w|=n}\dbZ/m\ar^{\psi}[d]\ar[r] &
\dbZ/m \ar@{=}[d] \\
N/N^m[S,N] \ar[u]^{i} & *-<3pc>{\times} & H^2(S/N) \ar[r]^{\qquad(\cdot,\cdot)'} & \,\dbZ/m,
}
\end{equation}
where $i$ is induced by inclusion (noting that $N^m[S,N]\leq S_{n+1,\dbZ/m}$, by the inclusion (\ref{inclusion of S_n})),
and $\psi((\bar r_w)_w)=\sum_w\bar r_w\psi_w$.
By Theorem \ref{fundamental duality}, the diagram commutes,
and by  (\ref{the pairing (.)'}),  the lower map is non-degenerate.
Lemma \ref{pairing of images} gives a non-degenerate bilinear map as in (a),
and it remains to show that it is perfect.

When the basis $A$ of $S$ is finite, the $\dbZ/m$-modules in the upper row of (\ref{main cd}) are finite.
Therefore so are the $\dbZ/m$-modules in the bilinear map (a), so its non-degeneracy implies its perfectness.

When $A$ is infinite we write $S=\invlim S_B$, where $B$ ranges over all finite subsets of $A$,
and let $\pi_B\colon S\to S_B$ be the associated projection.
Then
\[
NS_{n+1,\dbZ/m}/S_{n+1,\dbZ/m}=\invlim\pi_B(N)(S_B)_{n+1,\dbZ/m}/(S_B)_{n+1,\dbZ/m}
\]
and by the functoriality of the Massey product,
\[
H^2(S/N)_{n-\Massey}=\dirlim H^2(S_B/\pi_B(N))_{n-\Massey}.
\]
We now use the perfectness in the finite basis case and Lemma \ref{limits}.

\medskip

(b) \quad
As remarked above,  $N^m[S,N]\leq N\cap S_{n+1,\dbZ/m}$ and  $N\leq S^m[S,S]$.
We may now apply Proposition \ref{EM2 Prop 32} for
\[
G=S, \ T=N,\  T_0=N\cap S_{n+1,\dbZ/m}, \ H=H^2(S/N)_{n-\Massey}
\]
and use (a).
Note that $H^2(S)=0$.
\end{proof}

Taking here $N=S_{n,\dbZ/m}$ and $m=p$ prime, we get Theorem B.
See also \cite{Koch02}*{\S7.8} and \cite{NeukirchSchmidtWingberg}*{Prop.\ 3.9.13} for related facts in the case $n=2$.

From Theorem \ref{generalized Thm B ver1}(a) we recover the following result of Dwyer (see \cite{Dwyer75}*{Th.\ 3.1} for a more refined statement, in the discrete setting):

\begin{cor}
\label{gen of EM2}
Let $N,M$ be normal closed subgroups of $S$ with $N\leq M\leq S_{n,\dbZ/m}$.
The following conditions are equivalent:
\begin{enumerate}
\item[(a)]
$\inf\colon H^2(S/M)_{n-\Massey}\to H^2(S/N)_{n-\Massey}$ is an isomorphism;
\item[(b)]
$M\leq NS_{n+1,\dbZ/m}$.
\end{enumerate}
\end{cor}

From now on we assume that $m=p$ is prime.
Recall that $S_n=S_{n,\dbZ/p}$ (Proposition \ref{Jennings-Brauer}).

\begin{thm}
\label{generalized Thm B ver2}
\begin{enumerate}
\item[(a)]
$(\cdot,\cdot)'$ induces a non-degenerate bilinear map
\[
S_n/NS_{n+1} \times \Ker\bigl(H^2(S/S_n)_{n-\Massey}
\xrightarrow{\inf} H^2(S/N)\bigr)\to \dbZ/p.
\]
\item[(b)]
The kernels of the following inflation maps coincide:
\[
H^2(S/S_n)\to H^2(S/NS_{n+1}), \ H^2(S/S_n)_{n-\Massey}\to H^2(S/N).
\]
\end{enumerate}
\end{thm}
\begin{proof}
(a) \quad
 Theorem \ref{generalized Thm B ver1}(a) gives a commutative diagram of perfect pairings
\[
\xymatrix{
S_n/S_{n+1}  & *-<3pc>{\times} & H^2(S/S_n)_{n-\Massey} \ar[d]^{\inf}\ar[r] & \dbZ/p\ar@{=}[d]\\
NS_{n+1}/S_{n+1}\ar@{^{(}->}[u] & *-<3pc>{\times} & H^2(S/N)_{n-\Massey}\ar[r] &\dbZ/p.\\
}
\]
Since every $\dbF_p$-linear space is semi-simple, we may now apply Lemma \ref{pairing of Coker and Ker}.

\medskip

(b) \quad
This follows from (a) and Proposition \ref{EM2 Prop 32} with
\[
G=S/N, \ T=S_n/N,\  T_0=NS_{n+1}/N,\  H=H^2(S/S_n)_{n-\Massey}.
\qedhere
\]
\end{proof}

Now let $G$ be a profinite group and $\bar G=G(p)$ its maximal pro-$p$ quotient.
Suppose that $\bar G=S/N$ for a free profinite group $S$ and a closed normal subgroup $N$ of $S$
with $N\leq S_n$, $n\geq2$.

\begin{cor}
\label{thm B for G}
\begin{enumerate}
\item[(a)]
$(\cdot,\cdot)'$ induces a non-degenerate bilinear map
\[
G_n/G_{n+1} \times \Ker\bigl(H^2(G/G_n)_{n-\Massey}
\xrightarrow{\inf} H^2(G)\bigr)\to \dbZ/p.
\]
\item[(b)]
The kernels of the following inflation maps coincide:
\[
H^2(G/G_n)\to H^2(G/G_{n+1}), \ H^2(G/G_n)_{n-\Massey}\to H^2(G).
\]
\end{enumerate}
\end{cor}
\begin{proof}
By the 5-term exact sequence \cite{NeukirchSchmidtWingberg}*{Prop.\ 1.6.7}, $\inf\colon H^2(\bar G)\to H^2(G)$ is injective.
Since $G/G_k$ has a finite $p$-power exponent,  $G/G_k\isom\bar G/\bar G_k$ and
 $G_k/G_{k+1}\isom \bar G_k/\bar G_{k+1}$ canonically for every $k$.
Therefore we may replace $G$ by $\bar G$ to assume that $G=S/N$.
Let $\lam\colon S\to G$ be the projection map.
Then  $\lam(S_k)=G_k$ for every $k$.
Moreover, $S/S_n\isom G/G_n$ and $S_n/NS_{n+1}\isom G_n/G_{n+1}$.
We now apply Theorem \ref{generalized Thm B ver2}.
\end{proof}

See \cite{Bogomolov92}*{Lemma 3.3} in the case $n=2$.
Corollary \ref{gen of EM2} gives, by a similar reduction to the maximal pro-$p$ quotient:

\begin{cor}
\label{ststs}
Let $\bar M$ be a normal subgroup of $G$ contained in $G_n$.
Then $\inf\colon H^2(G/\bar M)_{n-\Massey}\to H^2(G)_{n-\Massey}$
is an isomorphism if and only if $\bar M\leq G_{n+1}$.
In particular, $H^2(G/G_{n+1})_{n-\Massey}\isom H^2(G)_{n-\Massey}$.
\end{cor}

Again, this follows from a profinite version of \cite{Dwyer75}*{Th.\ 3.1}.
For $n=2$ it was also proved in \cite{EfratMinac11b}*{Cor.\ 5.2, Th.\ A}, extending earlier results from \cite{CheboluEfratMinac12}.

\section{Proof of Theorem A'}
Let again $S=S_A$, $m=p$ prime, and $N$ a normal subgroup of $S$ contained in $S_n$.
The following theorem is an equivalent form of Theorem A' when we take $G=S/N$.
We note that while $H^2(S/S_n)_{n-\Massey}$ is generated by (elementary)  $n$-fold Massey products, this property
need not be  inherited by its subgroups.

\begin{thm}
\label{generalization of Thm A'}
Let $n\geq1$.
When $n\geq2$ we assume that
\[
\Ker\bigl(H^2(S/S_n)_{n-\Massey}\xrightarrow{\inf} H^2(S/N)\bigr)
\]
is generated by $n$-fold Massey products.
Then $NS_{n+1}=\bigcap M$, where $M$ ranges over
all open normal subgroups of $S$ containing $N$
such that $S/M$ embeds as a subgroup of $\U_{n+1}(\dbZ/p)$.
\end{thm}
\begin{proof}
We argue by induction on $n$.
For $n=1$ we have by Lemma \ref{description of S2}, $S_2=S^p[S,S]$, so $S/NS_2$ is an elementary abelian $p$-group.
Consequently, $NS_2=\bigcap M$, where $M$ ranges over all open normal subgroups of $S$ containing $N$
such that $S/M\isom\dbZ/p\isom\U_2(\dbZ/p)$.

Let $n\geq2$.
We assume the assertion for $n-1$ and prove it for $n$.

When $n\geq3$ Theorem \ref{generalized Thm B ver1}(b) implies that
\[
\inf\colon H^2(S/S_{n-1})_{(n-1)-\Massey}\to H^2(S/S_n)
\]
is the zero map, so trivially, its kernel is generated by $(n-1)$-fold Massey products.
We may therefore apply the induction hypothesis for $n-1$ and $N=S_n$ (also when $n=2$), to get
$S_n=\bigcap M$, where $M$ ranges over all open normal
subgroups of $S$ containing $S_n$ such that $S/M$ embeds as subgroup of $\U_n(\dbZ/p)$.

Next, we have already noted in the proof of Theorem \ref{generalized Thm B ver2}(b)  that (a) and (b) of
Proposition \ref{EM2 Prop 32}
hold with  $G=S/N$, $T=S_n/N$, $T_0=NS_{n+1}/N$,
$H=H^2(S/S_n)_{n-\Massey}$.
By assumption, the \textsl{set} $H_0$ of all  $n$-fold Massey products in
$\Ker(\inf\colon H\to H^2(G))$ generates this kernel.
Therefore (c) of Proposition \ref{EM2 Prop 32} also holds in this setup.

By Proposition \ref{upper triangular matrices},
the central extensions corresponding to $n$-fold Massey products in $H$ are (up to signs)
as in the lower part of the diagram
\[
\xymatrix{
&&& G\ar[ld]_{\hat\Phi}\ar[d] \\
0\ar[r] & \dbZ/p \ar[r] & \U_{n+1}(\dbZ/p)\times_{\bar \U_{n+1}(\dbZ/p)}(S/S_n)
\ar[r]^{\quad\qquad\qquad\pr_2}
& S/S_n\ar[r] &1,
}
\]
where the fiber product is with respect to the projection $\pi\colon \U_{n+1}(\dbZ/p)\to\bar \U_{n+1}(\dbZ/p)$
and some continuous homomorphism  $\bar\gam\colon S/S_n\to\bar \U_{n+1}(\dbZ/p)$.
The corresponding Massey product is in the kernel of
$\inf\colon H^2(S/S_n)\to H^2(G)$ (i.e., belongs to $H_0$)
if and only if there is a continuous homomorphism $\hat\Phi$ making
the diagram commutative \cite{Hoechsmann68}*{1.1}.
We therefore conclude from (c) of Proposition \ref{EM2 Prop 32} that
\[
NS_{n+1}/N=\bigcap\Ker(\hat\Phi),
\]
where $\hat\Phi$ ranges over all continuous homomorphisms making the diagram
commutative for some $\bar\gam$.
By Lemma \ref{kernels}, the subgroups $\Ker(\hat\Phi)$ are exactly the quotients
$(M\cap S_n)/N$,
where $M$ is a normal open subgroup of $S$ containing $N$ such that $S/M$ embeds in $\U_{n+1}(\dbZ/p)$.
Hence, $NS_{n+1}=(\bigcap M)\cap S_n$.
Since $\U_n(\dbZ/p)$ embeds as a subgroup of $\U_{n+1}(\dbZ/p)$, and by what we have seen
earlier in the proof, $S_n$ contains $\bigcap M$, so in fact $NS_{n+1}=\bigcap M$.
\end{proof}

\begin{rem}
\rm
In our general setup, the projection $S\to G$ induces an isomorphism $S/S_n\isom G/G_n$.
When it further induces an isomorphism $S/S_{n+1}\isom G/G_{n+1}$, i.e., $N\leq S_{n+1}$,
Theorem \ref{generalized Thm B ver1}(b) shows that the assumption about the kernel in
Theorem \ref{generalization of Thm A'}  is satisfied.
This example is not exhaustive (see Corollary \ref{absolute Galois groups}).
\end{rem}

\begin{rem}
\label{relaxed assumption}
\rm
As in Corollary \ref{thm B for G}, we may replace in Theorem A' the group $G$ by its maximal pro-$p$ quotient
$G(p)$ to assume that $G(p)$ (but not necessarily $G$) has  a presentation $S/N$ with $N\leq S_n$.
\end{rem}

\begin{cor}
\label{cd leq 1}
Let $G$ be a profinite group with $\cd_p(G)\leq1$.
Then $G_n=\bigcap_\rho\Ker(\rho)$, where $\rho$ ranges over all representations of $G$ in $\U_n(\dbF_p)$.
\end{cor}
\begin{proof}
The maximal pro-$p$ quotient $G(p)$ of $G$ is a free pro-$p$ group
\cite{NeukirchSchmidtWingberg}*{Prop.\ 3.5.3 and Prop.\ 3.5.9}.
We may therefore replace $G$ by $G(p)$, to assume that $G$ is a free pro-$p$ group.

As $H^2(G)=0$, all Massey kernel conditions are satisfied.
Take a free profinite group $S$ on the same basis as $G$ and let  $N=\Ker(S\to S(p)=G)$.
Since $S/S_n$ has a $p$-power exponent, $N\leq S_n$ for every $n$.
Now apply Theorem A'.
\end{proof}

\section{Examples}
\label{section on examples}
We conclude by examining Theorem A' in low degrees.

\begin{exam}
\rm
When $n=1$ Theorem A' is just the elementary fact that $G^p[G,G]=\bigcap M$,  where $M$ ranges over all
open normal subgroups of $G$ with $G/M\isom\{1\},\dbZ/p$.
\end{exam}

\begin{exam}
\rm
Let $n=2$.
Assume that the kernel of  $\inf\colon H^2(S/S_2)_{2-\Massey}=H^2(G/S_2)_\dec\to H^2(G)$
is generated by cup products (see Example \ref{cup products}).

When $p=2$, $\U_3(\dbZ/2)$ is the dihedral group $D_4$ of order $8$.
Hence in this case Theorem A' asserts that
\begin{equation}
\label{theorem from EM even case}
G_3=\bigcap\Bigl\{\bar M\ \bigm|\ \bar M\trianglelefteq G, \
G/\bar M\isom\{1\}, \dbZ/2, \dbZ/4, D_4\Bigr\}.
\end{equation}
This recovers \cite{EfratMinac11a}*{Cor.\ 11.3}, which in turn generalizes \cite{MinacSpira96}*{Cor.\ 2.18} (see below).

When $p>2$,  $\U_3(\dbZ/p)$ is the unique  nonabelian group $H_{p^3}$ of order $p^3$ and exponent $p$
(also called the \textsl{Heisenberg group}).
Its subgroups are $\{1\}$, $\dbZ/p$, $(\dbZ/p)^2$   and $\U_3(\dbZ/p)$ itself.
Moreover, when $G/\bar M\isom(\dbZ/p)^2$, we may write $\bar M=\bar M_1\cap\bar M_2$
with $G/\bar M_i\isom\dbZ/p$, $i=1,2$.
Therefore
\begin{equation}
\label{theorem from EM odd case}
G_3=\bigcap\Bigl\{\bar M\ \bigm|\ \bar M\trianglelefteq G, \
G/\bar M\isom\{1\}, \dbZ/p, H_{p^3}\Bigr\}.
\end{equation}
This recovers \cite{EfratMinac11b}*{Example 9.5(1)}.
\end{exam}

Finally we recover  \cite{MinacSpira96}*{Cor.\ 2.18} and \cite{EfratMinac11b}*{Th.\ D} (see also \cite{Thong12}):

\begin{cor}
\label{absolute Galois groups}
Let $K$ be a field containing a root of unity of order $p$ and $G=G_K$ its absolute Galois group.
Then  $G_3=\bigcap\bar M$, where $\bar M$
ranges over all open normal subgroups of $G$ such that $G/\bar M$ embeds in $\U_3(\dbZ/p)$.
\end{cor}
\begin{proof}
The injectivity of the  Galois symbol in degree $2$, which is a part of the Merkurev--Suslin theorem
(\cite{MerkurjevSuslin82}, \cite{GilleSzamuely}),
implies that the kernel of  $\inf\colon H^2(S/S^p[S,S])_\dec=H^2(G/G^p[G,G])_\dec\to H^2(G)$ is generated by cup products
 (see \cite{Bogomolov91},  \cite{EfratMinac11a}*{Prop.\ 3.2}).
Now apply Theorem A'.
\end{proof}

\begin{rem}
\rm
By \cite{EfratMinac11b} (see also \cite{CheboluEfratMinac12}*{Cor.\ 9.2}),  if $\bar S$ is a free pro-$p$ group, $N$ is a closed normal subgroup of $\bar S$ which is contained in $\bar S_3$, and $\bar S/N$ is not free pro-$p$, then $\bar S/N$ cannot be realized as the maximal pro-$p$ Galois group $G_F(p)$ of a field $F$ containing a root of unity of order $p$.
Therefore Theorem A' cannot be directly applied to describe $G_n$ for absolute Galois groups $G=G_F$ of fields $F$ as above when $n\geq4$ and $G$ is not free pro-$p$.
Yet, in the subsequent paper \cite{MinacTan13}, Min\'a\v c and T\^an show that the conclusion of Theorem A' holds in several other classes of groups, which \textsl{can} be realized as absolute Galois groups.
\end{rem}

\end{document}